\documentclass[12pt]{article}

\usepackage{amsmath,amssymb,amsthm}
\usepackage{mathrsfs}
\usepackage{authblk}

\usepackage[left=2cm,right=2cm,top=2cm,bottom=2cm,bindingoffset=0cm]{geometry}
\usepackage{nccfoots}
\usepackage[hang,flushmargin]{footmisc} 
\usepackage{hyperref}

\usepackage{graphicx}

\usepackage{pgf,tikz}

\newtheorem{remark}{Remark}
\newtheorem{example}{Example}

\newtheorem{proposition}{Proposition}
\newtheorem{lemma}{Lemma}

\newtheorem{theorem}{Theorem}

\newtheorem{definition}{Definition}

\DeclareMathOperator{\supp}{supp}
\DeclareMathOperator{\prob}{prob}

\title{Branching points in the planar Gilbert--Steiner problem have degree 3}

\author[1,2]{Danila Cherkashin}
\author[2,3]{Fedor Petrov}

\affil[1]{Institute of Mathematics and Informatics, Bulgarian Academy of Sciences, Sofia}
\affil[2]{Saint Petersburg State University}
\affil[3]{St. Petersburg Department of Steklov Mathematical Institute of Russian Academy of Sciences}

\begin{document}

\maketitle

\begin{abstract}
Gilbert--Steiner problem is a generalization of the Steiner tree problem on a specific optimal mass transportation.
    We show that every branching point in a solution of the planar Gilbert--Steiner problem has degree 3.
\end{abstract}

\section{Introduction}

One of the first models for branched transport was introduced by Gilbert~\cite{gilbert1967minimum}. The difference with the optimal transportation problem is that the extra geometric points may be of use; this explains the naming in honor of Steiner. Sometimes it is also referred to as \textit{optimal branched transport}; a large part of book~\cite{bernot2008optimal} is devoted to this problem. Let us proceed with the formal definition.

\begin{definition}
    Let $\mu^+,\mu^-$ be two finite measures on a metric space $(X,\rho(\cdot,\cdot))$ with finite supports such that total masses $\mu^+(X)=\mu^-(X)$ are equal. Let $V\subset X$  be a finite set containing the support of the signed measure $\mu^+-\mu^-$, the elements of $V$ are called \emph{vertices}. Further, let  $E$ be a finite collection of unordered pairs $\{x,y\}\subset V$ which we call \emph{edges}. So, $(V,E)$ is a  simple undirected finite graph. Assume that for every $\{x,y\}\in E$ two non-zero real numbers $m(x,y)$ and $m(y,x)$ are defined so that $m(x,y)+m(y,x)=0$. This data set is called a \emph{$(\mu^+,\mu^-)$-flow} if 
    \[
    \mu^+ - \mu^- = \sum_{\{x,y\}\in E} m(x,y)\cdot (\delta_y-\delta_x)
    \]
    where $\delta_x$ denotes a delta-measure at $x$ (note that the summand $m(x,y)\cdot (\delta_y-\delta_x)$ is well-defined in the sense that it does not depend on the order of $x$ and $y$).
\end{definition}

Let $C\colon [0,\infty)\to [0,\infty)$ be a \emph{cost function}. The expression
\[
\sum_{\{x,y\}\in E} C(|m(x,y)|) \cdot \rho(x,y)
\]
is called the \emph{Gilbert functional} of the $(\mu^+,\mu^-)$-flow.

The \textit{Gilbert--Steiner problem} is to find the flow which minimizes the Gilbert functional with cost function $C(x) = x^p$, for a fixed $p \in (0,1)$; we call a solution \emph{minimal flow}.

Vertices from $\supp(\mu^+) \setminus \supp(\mu^-)$ are called \textit{terminals}.
A vertex from $V \setminus \supp(\mu^+) \setminus \supp(\mu^-)$ is called a \textit{branching point}. Formally, we allow a branching point to have degree 2, but clearly it never happens in a minimal flow.

Local structure in the Gilbert--Steiner problem was discussed in~\cite{bernot2008optimal}, and the paper~\cite{lippmann2022theory} deals with planar case. A local picture around a branching point $b$ of degree 3 is clear due to the initial paper of Gilbert.
Similarly to the finding of the Fermat--Torricelli point in the celebrated Steiner problem one can determine the angles around $b$ in terms of masses (see Lemma~\ref{main_l}).

\begin{theorem}[Lippmann--Sanmart{\'\i}n--Hamprecht~\cite{lippmann2022theory}, 2022]
    \label{theorem:bezdari}
    A solution of the planar Gilbert--Steiner problem has no branching point of degree at least 5.
\end{theorem}

The goal of this paper is to give some conditions on a cost function under which all branching points in a planar solution have degree 3. They 
are slightly stronger than the Schoenberg \cite{schoenberg} conditions
of the embedding of the metric of the form 
$\rho(x,y):=f(x-y)$ to a Hilbert space. In particular, this covers the case of the standard cost function $x^p$, $0<p<1$. The following main theorem is the part of a more general Theorem~\ref{theorem:general}.

\begin{theorem}
 \label{theorem:main}
    A solution of the planar Gilbert--Steiner problem has no branching point of degree at least 4.
\end{theorem}

\section{Preliminaries}

We need the following lemmas. 

\begin{lemma}[Folklore]
Let $PQR$ be a triangle and $w_1$, $w_2$, $w_3$ be non-negative reals. 
For every point $X \in \mathbb{R}^2$ consider the value
\[
L(X) := w_1 \cdot  |PX| + w_2 \cdot  |QX| + w_3 \cdot  |RX|.
\]
Then 
\begin{itemize}
    \item [(i)] a minimum of $L(X)$ is achieved at a unique point $X_{min}$;
    \item [(ii)] if $X_{min} = P$ then $w_1 \geq w_2 + w_3$ or there is a triangle $\Delta$ with sides $w_1$, $w_2$, $w_3$ and $\angle P$ is at least the outer angle between $w_2$ and $w_3$ in $\Delta$.
\end{itemize}
\label{main_l}
\end{lemma}

Hereafter the metric space is the Euclidean plane $\mathbb{R}^2$. 

The following concept only slightly changes from that of
Schoenberg \cite{schoenberg}, introduced for describing which metrics
of the form $\rho(x,y)=f(x-y)$ on the real line can be embedded to
a Hilbert space. 

\begin{definition}
Let $\lambda$ be a Borel measure on $\mathbb{R}$ for which 
\begin{equation}\label{convergence}
\int \min(x^2,1)d\lambda(x)<\infty.
\end{equation}
Assume additionally that the support of $\lambda$ is uncountable. 
A function $f\colon \mathbb{R}_{\geqslant 0}\to \mathbb{R}_{\geqslant 0}$ of the form
\begin{equation}\label{def_of_adm_fun}
    f(t)=\sqrt{\int \sin^2 (tx)\, d\lambda(x)}=\frac12\|e^{2itx}-1\|_{L^2(\lambda)}
\end{equation}
is called \emph{admissible}.
\end{definition}

The only difference with \cite{schoenberg} is that we require 
that the support of the measure $\lambda$ is uncountable
which guarantees that the corresponding embedding has full dimension
(see below).

\begin{remark}
As $\lambda$ is a Borel measure, a continuous function
$\sin^2 (tx)$ is $\lambda$-measurable.
    Under conditions \eqref{convergence}, the integral in \eqref{def_of_adm_fun} 
    is finite, so $f(t)<\infty$ for all $t\geqslant 0$.  
\end{remark}
 
Further we are going to consider only admissible cost functions. Note that admissibility implies some properties one may expect from a cost function. In particular, $f(0)=0$ and 
$f$ is subadditive: for non-negative $t,s$ we have
\begin{align*}
    f(t)+f(s)&=\frac12\|e^{2itx}-1\|_{L^2(\lambda)} + \frac12\|e^{2isx}-1\|_{L^2(\lambda)} \\&=
\frac12\|e^{2itx}-1\|_{L^2(\lambda)} + \frac12\|e^{2i(s+t)x}-e^{2itx}\|_{L^2(\lambda)} 
\geq \frac12\|e^{2i(t+s)x} - 1\|_{L^2(\lambda)}=f(t+s).
\end{align*}
On the other hand it does not imply monotonicity (for instance, if $\supp \lambda \subset [0.9, 1.1]$ 
then $f(\pi) < f(\pi/2)$).

Hereafter $L^2(\lambda)$ for a measure $\lambda$ on $\mathbb{R}$ is understood as a \emph{real} Hilbert
space of complex-valued square summable w.r.t. $\lambda$ functions (strictly speaking, of classes of equivalences of such functions modulo coincidence $\lambda$-almost everywhere).

\begin{proposition}\label{independence}
If $\lambda$ is a Borel measure on $\mathbb{R}$ with uncountable support such that
\[
\int \min(x^2,1)d\lambda(x)<\infty,
\]
then any finite collection of functions of the form $e^{iax}-1$, $a\in \mathbb{R}$, is affinely independent in $L^2(\lambda)$.
\end{proposition}

\begin{proof}
    Assume the contrary. Then there exist distinct real numbers $a_1,\ldots,a_n$ and non-zero real coefficients $t_1,\ldots,t_n$ such that $\sum t_j=0$ and $\sum t_j(e^{ia_jx}-1)=0$ $\lambda$-almost everywhere. 
    But the analytic function $\sum t_j(e^{ia_jx}-1)$ is either identically zero, or has at most countably many (and separated) zeroes. In the latter case, it is not zero $\lambda$-almost everywhere, since the support of $\lambda$ is uncountable. The former case is not possible: indeed, if $\sum t_je^{ia_jx}\equiv 0$, then taking the Taylor expansion at 0 we get $\sum t_j a_j^k=0 $ for all $k=0,1,2,\ldots$. Therefore $\sum t_j W(a_j)=0$ for any polynomial $W$.
    Choosing $W(t)=\prod_{j=2}^n (t-a_j)$ we get $t_1=0$, a contradiction.
\end{proof}

One can see from the proof that the condition on uncountability of the support may be weakened.

\begin{lemma}\label{d1}
    Let $C$ be an admissible cost function.
    Define $h(m_1,m_2)$ as the value of the outer angle between $m_1$ and $m_2$ in the triangle with sides $C(|m_1|)$, $C(|m_2|)$, $C(|m_1+m_2|)$ (it exists by Proposition~\ref{independence}) for real $m_1,m_2$.
    Suppose that $OV_1$, $OV_2$ are edges in a minimal flow with masses $m_1$ and $m_2$. Then the angle between $OV_1$ and $OV_2$ is at least $h(m_1,m_2)$.
\end{lemma}

\begin{proof}
    Assume the contrary, then by Lemma~\ref{main_l} with $P = O$, $Q = V_1$, $R = V_2$, $w_1 = C(|m_1|)$, $w_2 = C(|m_2|)$, $w_3 = C(|m_1+m_2|)$ we have $X_{min} \neq O$. Then we can replace $[OV_1] \cup [OV_2]$ with $[X_{min}O] \cup [X_{min}V_1] \cup [X_{min}V_2]$ with the corresponding masses in our flow; this contradicts the minimality of the flow.
\end{proof}

\begin{lemma}\label{x_to_p}
For $0<p<1$, the function $f(x)=x^p$ is
admissible.    
\end{lemma}

\begin{proof}
    Consider the measure $d\lambda=x^{-2p-1}dx$ on $[0,\infty)$. Then $\int_0^\infty \min(x^2,1)d\lambda<\infty$ and for $t>0$ we have
    \[
    \int_0^\infty \sin^2(tx)\,d\lambda(x)=
    \int_0^\infty \sin^2(tx)\, x^{-2p-1}dx=
    t^{2p}\int_0^\infty \sin^2 y \, y^{-2p-1}dy,
    \]
    thus the measure $\lambda$ multiplied by an appropriate positive constant proves the result.
\end{proof}

\begin{example}
For another natural choice $d\lambda=4ce^{-2cx}dx$, $c>0$, we get 
an admissible function $f(t)=t/\sqrt{t^2+c^2}$.
\end{example}

The following lemma is essentially well-known, but for the sake of completeness and for covering degeneracies and the equality cases we provide a proof.

\begin{lemma}\label{f1}
Let $X$ be a finite-dimensional Euclidean space, let the points $A_0,A_1,A_2,\ldots,A_{n-1},A_n=A_0,A_{n+1}=A_1$ in $X$ be chosen so that $A_i\ne A_{i+1}$ for all $i=1,2,\ldots,n$. Denote $\varphi_i:=\pi-\angle A_{i-1}A_iA_{i+1}$ for $i=1,2,\ldots,n$.
Then $\sum \varphi_i\geqslant 2\pi$, and if the equality holds then the points $A_1,\ldots,A_n$ belong to the same two-dimensional affine plane.
\end{lemma}

\begin{proof}
    Let $u$ be a randomly chosen unit vector in $X$ (with respect to a uniform distribution on the sphere). For $j=1,2,\ldots,n$ denote by $U(j)$ the following event: $\langle u,A_j\rangle =\max_{1\leqslant i\leqslant n} \langle u,A_i\rangle$, where $\langle \cdot,\cdot\rangle$ denotes the inner product in $X$; and by $V(j)$ the event $\langle u,A_j\rangle =\max_{j-1\leqslant i\leqslant j+1} \langle u,A_i\rangle$. Obviously, $\prob U(j)\leqslant \prob V(j)$.
    Also, $\prob V(j)=\frac{\varphi_j}{2\pi}$, since the set of directions of $u$ for which $V(j)$ holds is the dihedral angle of measure $\varphi_j$. Thus, since always at least one event $U(j)$ holds, we get 
    \[
    1\leqslant \sum_{j=1}^n \prob U(j)\leqslant \sum_{j=1}^n \prob V(j)=\frac1{2\pi}\sum_{j=1}^n \varphi_j.
    \]
    This proves the inequality. It remains to prove that it is strict assuming that not all the points belong to a two-dimensional plane. Note that if every three consecutive points $A_{j-1},A_j,A_{j+1}$ are collinear, then all the points $A_1,\ldots,A_n$ are collinear that contradicts to our assumption. If $A_{j-1},A_j,A_{j+1}$ are not collinear, denote by $\alpha$ the two-dimensional plane they belong to. There exists $i$ for which $A_i\notin \alpha$. Then $\prob U(j)<\prob V(j)$, since there exist planes passing through $A_j$ which separate the triangle $A_{j-1}A_jA_{j+1}$ and the point $A_i$, and the measure of directions of such planes is strictly positive. Therefore, our inequality is strict.
\end{proof}

\section{Main result}

\begin{theorem} \label{theorem:general}
    Let $\mu^+,\mu^-$ be two measures with finite support on the Euclidean plane $\mathbb{R}^2$, and assume that the cost function $C$ is admissible. Then if a $(\mu^+,\mu^-)$-flow has a branching point of degree at least 4, then there exists a $(\mu^+,\mu^-)$-flow with strictly smaller value of Gilbert functional. 
\end{theorem}

\begin{proof} Assume the contrary. Let $O$ be a branching point, $OV_1,OV_2,\ldots,OV_k$, $k\geqslant 4$, be the edges incident to $O$, enumerated counterclockwise. Further the indices of $V_i$'s are taken modulo $k$, so that $V_1=V_{k+1}$ etc. Denote $m_i=m(OV_i)$, then by the definition of flow we get $\sum m_i=0$. By Lemma~\ref{d1}, $\angle V_iOV_{i+1}\geqslant h(m_i,m_{i+1})$. 

Consider the functions $A_j(x):=e^{i(m_1+\ldots+m_j)x}-1$ for $j=1,2,\ldots$ (here $i$ is the imaginary unit). 
Then $\sum m_j=0$ yields that  $A_{j+k}\equiv A_j$ for all $j>0$.

Since the cost function $C(t)$ is admissible, there exists a Borel measure $\lambda$ on $\mathbb{R}$ with uncountable support such that  $\int \min(x^2,1)d\lambda(x)<\infty$ and 
\[
C(t)=\sqrt{\int 4\sin^2 \frac{tx}2 d\lambda(x)}.
\]

Using the identity $|e^{ia}-e^{ib}|^2=4\sin^2\frac{a-b}2$ for real $a,b$ we note that for $j,s>0$ in the Hilbert space $L^2(\lambda)$ we have
\[
\|A_{j+s}-A_j\|^2 = C(|m_{j+1}+\ldots+m_{j+s}|)^2.
\]
In particular, the lengths of the sides of the triangle $A_{j-1}A_jA_{j+1}$ are equal to $C(|m_j|)$, $C(|m_{j+1}|)$ and $C(|m_{j}+m_{j+1}|)$. 
Therefore $\varphi_j:=\pi-\angle A_{j-1}A_jA_{j+1} = h(m_j,m_{j+1})$. 
By Lemma \ref{f1} we get $\sum \varphi_j\geqslant 2\pi$.

By Lemma \ref{main_l}, this yields $2\pi=\sum_{j=1}^k \angle V_jOV_{j+1}\geqslant \sum \varphi_j \geqslant 2\pi$. Therefore, the equality  must take place. Again by Lemma \ref{f1} it follows that the points $A_j$ belong to the same 2-dimensional subspace. 
But by Proposition \ref{independence}, distinct points between $A_j$'s are affinely independent.
Therefore, there exist at most three distinct $A_j$'s, and if exactly three, they are not collinear. 
It is easy to see that the
equality $\sum \varphi_j=2\pi$ under these conditions does not hold when $k>3$. A contradiction.
\end{proof}

\section{Examples of branching points of degree 4}

Let us start with an example in three dimensions. Consider four masses $m_1,m_2,m_3,m_4$ of zero sum, such that no two of them give zero sum.
Repeat the beginning of the proof of Theorem~\ref{theorem:main} to get the simplex $A_1A_2A_3A_4$ in 3-dimensional space.
Now consider unit edges $OB_i$ in $\mathbb{R}^3$ with directions $A_{i-1}A_{i}$, $1\leq i\leq 4$.
By the construction the angles between vectors $OB_i$ and $OB_{i+1}$ are exactly $h(m_i,m_{i+1})$. 
Suppose that angles $\angle B_1OB_3$ and $\angle B_2OB_4$ are greater than $h(m_1,m_3)$ and $h(m_2,m_4)$, respectively.

Then we claim that the flow 
\[
\sum_{i=1}^4 m_i \cdot (\delta_O - \delta_{B_i})
\]
is the unique solution of the corresponding Gilbert--Steiner problem.

First, if we fix the graph structure then the position of $O$ is optimal by the following lemma, because the closeness of the polychain
$A_1A_2A_3A_4A_1$ gives exactly~\eqref{eq:geometricmean}.

\begin{lemma}[Weighted geometric median,~\cite{weiszfeld2009point}]
Consider different non-collinear points $A,B,C,D \in \mathbb{R}^3$ and let $w_1$, $w_2$, $w_3$, $w_4$ be non-negative reals. 
Then 
\[
L(X) := w_1 \cdot  |AX| + w_2 \cdot  |BX| + w_3 \cdot  |CX| + w_4 \cdot |DX|
\]
has unique local (and global) minimum satisfying
\begin{equation}\label{eq:geometricmean}
w_1 \bar{e_A} + w_2 \bar{e_B} + w_3 \bar{e_C} + w_4 \bar{e_D} = 0,    
\end{equation}
where $\bar{e_A},\bar{e_B},\bar{e_C},\bar{e_D}$ are unit vectors codirected with $XA$, $XB$, $XC$, $XD$, respectively.
\label{main_l2}
\end{lemma}

Since any two masses have nonzero sum, every flow is connected.
Thus every possible competitor has 2 branching points of degree 3. Consider the case in which branching points $U$ and $V$ are connected with $B_1,B_2$ and $B_3,B_4$, respectively.
By the convexity of length, the Gilbert functional $L(U,V)$ considered on the set of all possible $U$ and $V$ ($\mathbb{R}^3 \times \mathbb{R}^3$) is a convex function. Let us show that $U = V = O$ is a local minimum.
Indeed, consider $U_\varepsilon = O + \varepsilon u$ and $V_\delta = O + \delta v$ for arbitrary unit vectors $u,v$ and small positive $\varepsilon, \delta$.
Then 
\[
L(U_\varepsilon,V_\delta) - L(O,O) = 
\]
\[
(1+o(1)) \cdot \left(w(UV)\cdot \| \varepsilon u - \delta v \| - \varepsilon \langle w_1e_1,u \rangle 
 - \varepsilon \langle w_2e_2,u\rangle  - \delta \langle w_3e_3,v \rangle  - \delta \langle w_4e_4,v \rangle  \right)=
\]
\[
(1+o(1)) \cdot \left(w(UV)\cdot \|\varepsilon u - \delta v\| - \varepsilon \langle w_{12}e_{12},u \rangle - \delta \langle w_{34}e_{34},v \rangle \right),
\]
where $w_{12}e_{12} = w_1e_1 + w_2e_2$ and $w_{34}e_{34} = w_3e_3 + w_4e_4$ for unit $e_{12}$ and $e_{34}$. By the construction one has $w_{12} = w_{34} = w(UV)$ and $e_{12} + e_{34} = 0$, so
\[
L(U_\varepsilon,V_\delta) - L(O,O) = (1+o(1)) \cdot w(UV)\cdot (\|\varepsilon u - \delta v\| - \langle e_{12}, \varepsilon u - \delta v\rangle ).
\]
Since $e_{12}$ is unit, the derivative is non-negative for every $u,v$. 

The case in which $U$ and $V$ are connected with $B_2,B_3$ and $B_4,B_1$, respectively, is completely analogous. In the remaining case ($U$ is connected with $B_1,B_3$ and $V$ is connected with $B_2,B_4$) we have $w_{12} = w_{34} < w(UV)$ due to $\angle B_1OB_3 > h(m_1,m_3)$ and $\angle B_2OB_4 > h(m_2,m_4)$. Thus $U = V = O$ is also a local minimum.

It is known~\cite{gilbert1967minimum} that $L$ has a unique local and global minimum, which finishes the example.

No proceed with planar examples of $4$-branching for some non-admissible cost-function $C$.
Then we may repeat the 3-dimensional argument starting with planar $A_1A_2A_3A_4$. 

The simplest way to produce an example is to consider a trapezoid $A_1A_2A_3A_4$ and apply Ptolemy's theorem.
This case corresponds to $m_1 = m_3$ and $m_1 + m_2 + m_3 + m_4 = 0$.
Then $|A_1A_2| = |A_3A_4| = C(|m_1|)$, $|A_2A_3| = C(|m_2|)$, $|A_4A_1| = C(|m_4|)$ and $|A_1A_3| = |A_2A_4| = C(|m_1 + m_2|)$.
The existence of such trapezoid means
\begin{equation} \label{ptolemy}
C(|m_1 + m_2|)^2 = C(|m_1|) \cdot \big ( C(|m_2|) + C(|2m_1 + m_2|) \big).    
\end{equation}
If we assume that $C$ is monotone and subadditive then~\eqref{ptolemy} means that a trapezoid exists; note that we need values of $C$ only at 4 points.

Now we give an example of a monotone, subadditive and concave cost function with 4-branching. 
For this purpose put $m_1 = m_2 = m_3 = 1$ and $m_4 = -3$, $C(1) = 1$, $C(2) = 1.9$, $C(3) = 2.61$; clearly~\eqref{ptolemy} holds.
Now one can easily interpolate a desired $C$, for instance
\[
C(t) = \begin{cases}
    t , & t \leq 1 \\
    0.1 + 0.9 t, & 1 < t \leq 2\\
    0.48 + 0.71t, & 2 < t \leq 3\\
    1.11 + 0.5t & 3 < t.
\end{cases}
\]
Finally, the inequalities $\angle B_1OB_3 > \angle B_2OB_3 = h(m_2,m_3) = h(m_1,m_3)$ and $\pi = \angle B_2OB_4 > h(m_2,m_4)$ hold.

\section{Open questions}

%Our definition of admissible functions is motivated by the famous Bochner characterization of positive definite kernels. 
It would be interesting to describe all cost functions for which the conclusion of Theorem~\ref{theorem:general} holds.

Now let us focus on the cost function $C(x) = x^p$. Having a knowledge that every branching point has degree 3 one can adapt Melzak algorithm~\cite{melzak1961problem} from Steiner trees to Gilbert--Steiner problem. 
The idea of the algorithm is that after fixing the combinatorial structure one can find two terminals $t_1$, $t_2$ connected with the same branching point $b$. Then one may reconstruct the solution for $V$ from the solution for $V \setminus \{t_1,t_2\} \cup \{t'\}$  for a proper $t'$ which depend only on $t_1,t_2$ (in fact one has to check 2 such $t'$). When the underlying graph is a matching we finish in an obvious way. 
Application of this procedure for all possible combinatorial structures gives a slow but mathematically exhaustive algorithm in the planar case.

However there is no known algorithm in $\mathbb{R}^d$ for $d > 2$ (see Problem 15.12 in~\cite{bernot2008optimal}). Recall that we have to consider a high-degree branching.

A naturally related problem is to evaluate the maximal possible degree of a branching point in the $d$-dimensional Euclidean space for every $d$. Note that the dependence on the cost function may be very complicated.

Some other questions are collected in Section~15 of~\cite{bernot2008optimal} (some of them are solved, in particular Problem 15.1 is solved in~\cite{colombo2021well}).

\paragraph{Acknowledgements.} The research is supported by RSF grant 22-11-00131. We thank Guy David and Yana Teplitskaya for sharing the problem. We are grateful to Vladimir Zolotov for reading a previous version of the text and 
pointing out our attention at~\cite{schoenberg}.

\bibliography{main}

\begin{thebibliography}{1}

\bibitem{bernot2008optimal}
Marc Bernot, Vicent Caselles, and Jean-Michel Morel.
\newblock {\em Optimal transportation networks: models and theory}.
\newblock Springer, 2008.

\bibitem{colombo2021well}
Maria Colombo, Antonio De~Rosa, and Andrea Marchese.
\newblock On the well-posedness of branched transportation.
\newblock {\em Communications on Pure and Applied Mathematics}, 74(4):833--864,
  2021.

\bibitem{gilbert1967minimum}
Edgar~N. Gilbert.
\newblock Minimum cost communication networks.
\newblock {\em Bell System Technical Journal}, 46(9):2209--2227, 1967.

\bibitem{lippmann2022theory}
Peter Lippmann, Enrique Fita~Sanmart{\'\i}n, and Fred~A. Hamprecht.
\newblock Theory and approximate solvers for branched optimal transport with
  multiple sources.
\newblock {\em Advances in Neural Information Processing Systems}, 35:267--279,
  2022.

\bibitem{melzak1961problem}
Zdzislaw~A. Melzak.
\newblock On the problem of {S}teiner.
\newblock {\em Canadian Mathematical Bulletin}, 4(2):143--148, 1961.

\bibitem{schoenberg}
Isaak~J. Schoenberg.
\newblock Metric spaces and positive definite functions.
\newblock {\em Transactions of the American Mathematical Society},
  44(3):522--536, 1938.

\bibitem{weiszfeld2009point}
Endre Weiszfeld and Frank Plastria.
\newblock On the point for which the sum of the distances to $n$ given points
  is minimum.
\newblock {\em Annals of Operations Research}, 167:7--41, 2009.

\end{thebibliography}
\bibliographystyle{plain}

\end{document}